\documentclass{amsart}

\usepackage[utf8]{inputenc}
\usepackage{amssymb}
\usepackage{euscript}
\usepackage{mathrsfs}
\usepackage{a4wide}
\usepackage{setspace}

\makeatletter
\@namedef{subjclassname@2010}{%
\textup{2010} Mathematics Subject Classification}
\makeatother

\numberwithin{equation}{section}

\newtheorem{thm}{Theorem}[section]
\newtheorem{cor}[thm]{Corollary}
\newtheorem{lem}[thm]{Lemma}
\newtheorem{pro}[thm]{Proposition}

\theoremstyle{remark}

\theoremstyle{definition}
\newtheorem{exa}[thm]{Example}

\newcommand*{\C}{\mathbb{C}}
\newcommand*{\R}{\mathbb{R}}
\newcommand*{\Z}{\mathbb{Z}}
\newcommand*{\N}{\mathbb{N}}
\newcommand*{\REP}{\overline{\mathbb{R}}_+}

\newcommand{\D}{\mathrm{d}}
\newcommand*{\trc}[1]{\mathrm{tr}(#1)}
\newcommand{\ball}[1]{\mathrm{B}_{#1}}

\newcommand*{\borel}[1]{\mathscr{B}(#1)}
\newcommand*{\dom}[1]{\mathscr{D}(#1)}

\newcommand*{\PLN}{\mathscr{P}}

\newcommand*{\hh}{\mathcal H}

\newcommand*{\ogr}[1]{\mathbf{B}(#1)}
\newcommand*{\is}[2]{\langle #1, #2\rangle}
\newcommand*{\hsn}[1]{\|#1\|_{\mathrm{HS}}}
\newcommand*{\hs}[1]{\mathbf{HS}(#1)}

\newcommand*{\TP}[1]{\mathsf{T}_{#1}}
\newcommand*{\IGR}[1]{\mathsf{J}_{#1}}
\newcommand*{\PROJ}{\mathsf{P}}
\newcommand*{\SB}[2]{{#1}^2(\gm)\otimes #2}

\newcommand*{\gm}{\mu_{\mathbf{G}}}
\newcommand*{\KHH}{\mathscr{K}\otimes \hh}

\title[Compact vectorial Toeplitz operators]{Compact vectorial Toeplitz operators on the Segal-Bargmann space}

\author{Tomasz Beberok}
\address{Department of Applied Mathematics, University of Agriculture in Krakow, ul. Balicka 253c, 30-198 Krakow, Poland.}
\email{tomasz.beberok@ur.krakow.pl}

\author{Piotr Budzyński}
\address{Department of Applied Mathematics, University of Agriculture in Krakow, ul. Balicka 253c, 30-198 Krakow, Poland.}
\email{piotr.budzynski@ur.krakow.pl}

\author{Dong-O Kang}
\address{Department of Mathematics, Chungnam National University, 30 Daejeon, Korea}
\email{dokang@hcnuac.kr}

\begin{document}
\setstretch{1.1}
\maketitle
\begin{abstract}
We provide a sufficient condition for the compactness of a Toeplitz operator acting on the Segal-Bargmann space of vector-valued functions written in terms of an associated operator-valued kernel.
\end{abstract}
\section{Introduction}
A well-known results due to Stroethoff states that a Toeplitz operator $T_\phi$, with $\phi\in L^\infty(\C^n)$, acting on the Segal-Bargmann space $\mathcal{B}$ is compact if and only if $\|\mathrm{P}(\phi\circ\tau_\lambda)\|\to0$ as $|\lambda|\to\infty$, where $\tau_\lambda$ is the translation on $\C^n$ by $\lambda$ and $\mathrm{P}$ is the orthogonal projection from $L^2(\gm)$ onto $\mathcal{B}$ (see \cite[Theorem 5]{str-mmj-1992}). This result is related to a characterization of compact multiplication operators in terms of Berezin symbols due to Berger and Coburn (see \cite[Theorem C]{ber-cob-tams-1987}). The paper aims at generalizing Stroethoff's result to the context of Toeplitz operators acting on the Segal-Bargmann space of vector-valued functions.

Toeplitz operators on Segal-Bargmann spaces arise in quantization of classical mechanics  and are related to pseudodifferential operators (see \cite{bar-cpam-61, gro-lou-ste-aif-68, gui-ieot-84, how-jfa-80, new-sha-bams-66, new-sha-ams-68, seg}). They have been studied since the work of Berezin (see \cite{ber-mi-1974, ber-mi-1975}) in the classical context of complex-valued entire functions (see, e.g., \cite{bau-lee-jfa-2011, bau-cob-isr-jfa-2010, cob-isr-li-tams-2011,eng-jmaa-2010,isr-zhu-ieot-2010, jan-sm-1991, jan-sto-jfa-1994}) and also in the more general setting of vector-valued functions (see, e.g., \cite{cic-ieot-1999, cic-sm-2002,cic-sha-ms-2003}). The literature concerning these operators is broad and still growing.

The main result of the paper states that a vectorial Toeplitz operator $\TP{\Phi}$ with $L^\infty$ symbol $\Phi$ is compact whenever its "oscilation at infinity" is small (see Theorem \ref{main}). The condition for that, namely condition \eqref{stendahl}, is a generalization of the Stroethoff's one and employs the Hilbert-Schmidt norm. The reason for using the Hilbert-Schmidt norm lies in the fact that the compactness of $\TP{\Phi}$ is shown by an approximating its adjoint, proved to be a vectorial integral operator with kernel related to translations (see Proposition \ref{lem2}), with a sequence of Hilbert-Schmidt vectorial integral operators. To this end we adapt a method used in the classical context (see Lemma \ref{inths}, Proposition \ref{inths4new}, and Corollary \ref{venti}). It remains an open question whether condition \eqref{stendahl} is necessary for the compactness of $\TP{\Phi}$. A necessary condition we obtained (see Proposition \ref{necc}) turns out to be essentially weaker (see Examples \eqref{star1}, \eqref{star2}, and \ref{taebaek}).
\section{Preliminaries}
In all what follows $\R$ and $\C$ stand for the sets of real and complex numbers, respectively, $\N$ denotes the set of all natural numbers $\{1,2,3,\ldots\}$, $\Z_+=\N\cup\{0\}$, $\R_+=[0,\infty)$, and $\REP=\R_+\cup\{\infty\}$. For $R\in(0,\infty)$, $\ball{R}$ stands for the ball $\{z\in \C^n\colon |z|\leqslant R\}$, where $|\cdot|$ is the standard Euclidean norm in $\C^n$. The following standard multiindex notation is used: $z^j=z_1^{j_1}\cdots z_n^{j_n}$, $j!=j_1!\cdots j_n!$, and $|j|=j_1+\ldots+j_n$ for all $j=(j_1,\ldots,j_n)\in\Z_+^n$ and $z=(z_1,\ldots,z_n)\in\C^n$, $n\in\N$.

Let $\hh$ be a (complex) Hilbert space, with the inner product  $\is{\cdot}{-}$ and the corresponding norm $\|\cdot\|$. Let $A$ be a (linear) operator in $\hh$. Then $\dom{A}$ stands for the domain of $A$ and $A^*$ stands for the adjoint of $A$, whenever it exists. $\ogr{\hh}$ denotes the algebra of all bounded operators on $\hh$ with the standard operator norm $\|\cdot\|_{\ogr{\hh}}$. If this leads to no confusion we write $\|\cdot\|$ instead of $\|\cdot\|_{\ogr{\hh}}$. The Banach space of all the Hilbert-Schmidt operators equipped with by Hilbert-Schmidt norm $\hsn{\cdot}$ is denoted by $\hs{\hh}$. Throughout the paper we adhere to the convention that if $A$ is not a Hilber-Schmidt operator, then $\hsn{A}=\infty$.

Let $\nu$ be a positive measure on $\borel{\C^n}$, the $\sigma$-algebra of Borel sets in $\C^n$. Then as usual $L^2(\nu)$ stands for the space of all square-summable (with respect to $\nu$) complex-valued Borel functions on $\C^n$. By $\nu\otimes\nu$ we denote the product measure on $\borel{\C^n\times\C^n}=\borel{\C^{2n}}$. 

Let $\gm$ be the $n$-dimensional Gaussian measure on $\C^n$, i.e.,
\begin{align*}
\gm(\sigma)=\frac{1}{\pi^n}\int_{\sigma}\exp{\big(-|z|^2\big)}\D V(z),\quad \sigma\in\borel{\C^n},
\end{align*}
where $V$ denotes the Lebesgue measure on $\borel{\C^n}$. The {\em Segal-Bargmann space} $\mathcal{B}$  (a.k.a. the {\em Fock space}) is a closed subspace of $L^2(\gm)$ consisting of all analytic functions belonging to $L^2(\gm)$. It is well-known that $\mathcal{B}$ is a RKHS with the kernel $k(z,w)=\exp{\is{z}{w}}$, $z,w\in \C^n$. Given $\lambda \in \C^n$, we define $k_\lambda(z)=\exp{\is{z}{\lambda}}$, $z\in\C^n$.

Let $\hh$ be a separable Hilbert space. The space of all analytic functions $F\colon \C^n\to \hh$ such that $\int_{\C^n}\|F(z)\|^2\D\gm(z)<\infty$ is to be identified with $\mathcal{B}\otimes \hh$, the Hilbert tensor product of $\mathcal{B}$ and $\hh$. We call it a {\em (vectorial) Segal-Bargmann} space. 
For $f\in\mathcal{B}$ and $h\in\hh$, $f\otimes h$ stands for the function defined as $(f\otimes h) (z)=f(z)h$, $z\in\C^n$. Then $\KHH$ stands for the linear span of $\{k_\lambda\otimes g\colon \lambda\in\C^n, g\in\hh\}$. The Segal-Bargmann space has the following reproducing property:
\begin{align}\label{reprod}
\is{F}{k_z\otimes h}=\is{F(z)}{h},\quad z\in\C^n, h\in\hh, F\in\mathcal{B}\otimes \hh.
\end{align}
Now, given a Borel function $\Phi\colon \C^n\to \ogr{\hh}$, we define the {\em (vectorial) Toeplitz operator}
\begin{align*}
\TP{\Phi}\colon \mathcal{B}\otimes \hh\supseteq\dom{\TP{\Phi}}\to\mathcal{B}\otimes \hh  \end{align*}
by the formula
\begin{align*}
\dom{\TP{\Phi}}&=\big\{F\in \mathcal{B}\otimes \hh\colon \Phi F\in \SB{L}{\hh}\big\},\\
\TP{\Phi} F&= \PROJ(\Phi F),\quad F\in\dom{\TP{\Phi}},
\end{align*}
where $\PROJ$ denotes the orthogonal projection from $\SB{L}{\hh}$ onto $\mathcal{B}\otimes \hh$.
Clearly, due to the reproducing property \eqref{reprod} we have
\begin{align}\label{toeplitzform}
\is{(\TP{\Phi}F) (z)}{h}=\int_{\C^n}\is{\Phi(w)F(w)}{h}\overline{k_z(w)}\D\gm(w),\quad z\in \C^n, F\in\dom{\TP{\Phi}}, h\in\hh.
\end{align}

Given a {\em $\ogr{\hh}$-valued kernel on $\C^n$}, i.e., a Borel function $\Theta\colon \C^n\times \C^n\to \ogr{\hh}$, we define the  {\em (vectorial) integral operator} \begin{align*}
\IGR{\Theta}\colon \mathcal{B}\otimes \hh\supseteq \dom{\IGR{\Theta}}\to \mathcal{B}\otimes \hh
\end{align*}
with the following formula
\begin{align*}
\dom{\IGR{\Theta}}&=\big\{F\in \mathcal{B}\otimes \hh\colon J_\Theta F\in \SB{L}{\hh}\big\},\\
\IGR{\Theta} F&= \PROJ\big( J_\Theta F\big),\quad F\in\dom{\IGR{\Theta}},
\end{align*}
where 
\begin{align*}
\big(J_\Theta F\big) (z) :=\int_{\C^n} \Theta(z,w)F(w)\D\gm(w),\quad z\in\C^n, F\in\mathcal{B}\otimes \hh \end{align*}
is understood in the weak sense. The kernel $\Theta$ is called {\em Hilbert-Schmidt} if it satisfies the condition
\begin{align}\label{hskernel}
\int_{\C^n}\int_{\C^n} \hsn{\Theta(z,w)}^2\D\gm(z)\D\gm(w)<\infty.
\end{align}
Let us recall a vectorial counterpart of the well-known classical result considering integral operators (see \cite[p. 253]{bir-sol}). For completeness and the reader convenience we supply a proof (different one than that in \cite{bir-sol}).
\begin{lem}\label{inths}
Let $\Theta\colon \C^n\times\C^n\to\ogr{\hh}$ be a Hilbert-Schmidt kernel. Then $\IGR{\Theta}$ is a Hilbert-Schmidt operator on $\mathcal{B}\otimes \hh$.
\end{lem}
\begin{proof}
Since $\|A\|\leqslant \hsn{A}$ for every Hilbert-Schmidt operator $A$, the Cauchy-Schwarz inequality and \eqref{hskernel} imply that the operator $\IGR{\Theta}$ belongs to $\ogr{\mathcal{B}\otimes \hh}$.

Due to \eqref{hskernel} and the fact that $\hs{\hh}$ is separable, the kernel $\Theta$ is belongs to the Bochner-Lebesgue space $L^2\big(\gm\otimes\gm; \hs{\hh}\big)$. It is well-known that the the set of all simple functions is dense in $L^2\big(\gm\otimes\gm; \hs{\hh}\big)$, thus there exists a sequence of simple kernels $\{\Theta_j\}_{j=1}^\infty$ converging to $\Theta$ in $L^2\big(\gm\otimes\gm; \hs{\hh}\big)$ as $n\to\infty$. Each of the kernels $\Theta_j$, $\in\N$, has the form $\Theta_j=\sum_{l=1}^{N_j}\chi_{\varOmega_l} S_l$ with 
$\{\varOmega_l\}_{l=1}^{N_j}\subseteq\borel{\C^n\times\C^n}$ and $\{S_l\}_{l=1}^{N_j}\subseteq\hs\hh$. Since for all $\varOmega\subseteq\borel{\C^n\times \C^n}$ and $S\in\hs{\hh}$ we have
\begin{align*}
\IGR{\Xi} F(z)=\int_{\C^n}\Xi(z,w)F(w)\D\gm(w)=S\int_{\C^n}\chi_{\varOmega}(z,w)F(w)\D\gm(w),\quad F\in\mathcal{B}\otimes \hh,      
\end{align*}
where $\Xi(z,w)=\chi_\varOmega(z,w) S$, we see that for every $j\in\N$, the operator $\IGR{\Theta_j}$ is Hilbert-Schmidt and $\|\IGR{\Theta_j}\|_{\hs\hh}^2=\sum_{l=1}^{N_j}\|S_l\|_{\hs{\hh}}^2(\gm\otimes\gm)(\varOmega_l)=\|\Theta_j\|_{L^2(\gm\otimes\gm; \hs{\hh})}^2$. This implies convergence of $\{\IGR{\Theta_j}\}_{j=1}^\infty$ in $\hs{\mathcal{B}\otimes\hh}$. On the other hand, for every $F\in\mathcal{B}\otimes \hh$ we have\allowdisplaybreaks
\begin{align*}
\|(\IGR{\Theta_j}-\IGR{\Theta}) F\|^2
&\leqslant\int_{\C^n} \big\| \int_{C_n}\big(\Theta_j(z,w)-\Theta_j(z,w)\big)F(w)\D\gm(w)\big\|^2\D\gm(z)\\
&\leqslant\int_{\C^n} \bigg( \int_{C_n}\|\Theta_j(z,w)-\Theta_j(z,w)\|\|F(w)\|\D\gm(w)\bigg)^2\D\gm(z)\\
&\leqslant \|F\|^2 \int_{\C^n} \int_{C_n}\hsn{\Theta_j(z,w)-\Theta_j(z,w)}^2\D\gm\otimes\gm(z,w)\\
&= \|F\|^2 \|\Theta_j-\Theta\|_{L^2(\gm\otimes\gm;\hs{\hh})}^2.
\end{align*}
Hence we deduce that $\{\IGR{\Theta_j}\}_{j=1}^\infty$ converges to $\IGR{\Theta}$ in $\hs\hh$, which completes the proof.
\end{proof}

As shown below, the adjoint of a Toeplitz operator can be recovered from a integral operator with a kernel that is associated to the symbol of the Toeplitz operator via translations. For that we recall a translation-related formula. Given $\lambda\in\C^n$, the translation $\tau_\lambda\colon \C^n\to \C^n$ is given by $\tau_\lambda(w)=w+\tau$, $w\in \C^N$. Using the change-of-variable formula, we get
\begin{align}\label{COV}
\int_{\C^n}(f\circ \tau_\lambda)(w) \D\gm(w)=\int_{\C^n} f(w) \widetilde k_\lambda(w)\D\gm
\end{align}
with 
\begin{align*}
\widetilde k_\lambda(w)=\frac{|k_\lambda(w)|^2}{k_\lambda(\lambda)},\quad w\in \C^n,
\end{align*}
holding for every Borel function $f\colon \C^n\to \REP$ or every $f\in L^2\big(\widetilde k\,\D\gm\big)$.

Here, and later on, given $\Phi\colon\C\to\ogr{\hh}$ and $g\in\hh$, we denote by $\Phi\otimes g$ the function $\C^n\to\hh$ defined by $(\Phi\otimes g)(z)=\Phi(z)g$.

Toeplitz operators act on functions $k_\lambda\otimes g$, with $\lambda\in\C^n$ and $g\in\hh$, in accordance to the following formula involving translations (cf. \cite[Proposition 1]{str-mmj-1992}). 
\begin{lem}\label{lem1}
Let $\lambda\in \C^n$ and $g\in\hh$ be such that $k_\lambda\otimes g\in \dom{\TP{\Phi}}$. Then
\begin{align*}
{(\TP{\Phi}k_\lambda\otimes g) (z)}={k_\lambda(z) \big(\PROJ( \Phi_\lambda\otimes g)\big)(\tau_{-\lambda }(z))}\quad z\in\C^n,
\end{align*}
where  $\Phi_\lambda=\Phi\circ\tau_\lambda$.
\end{lem}
\begin{proof}
Fix $z\in\C^n,\ h\in \hh$. Since $k_\lambda\otimes g\in \dom{\TP{\Phi}}$, we see that \eqref{COV} implies \allowdisplaybreaks
\begin{align*}
\int_{\C^n}\|\Phi_\lambda(u)g\|^2\D\gm(u)
&=\int_{\C^n}\|\Phi(u)g\|^2\widetilde k_\lambda(u)\D\gm(u)\\ &=\frac{1}{k_\lambda(\lambda)}\int_{\C^n}\|\Phi(u)g\|^2|k_\lambda(u)|^2\D\gm(u)\\
&=\frac{1}{k_\lambda(\lambda)}\int_{\C^n}\|\Phi(u)k_\lambda(u)g\|^2\D\gm(u)\\
&=\frac{1}{k_\lambda(\lambda)}\int_{\C^n}\|\Phi(u)(k_\lambda\otimes g)(u)\|^2\D\gm(u)<\infty,
\end{align*}
which means that $u\mapsto \Phi_\lambda(u)g$ belongs to $\SB{L}{\hh}$. 
Now, it is easily seen that (cf. \cite{str-mmj-1992})
\begin{align*}
k_\lambda(u)\overline{k_z(u)}=\widetilde k_\lambda(u) k_\lambda(z) \overline{k_{\tau_{-\lambda}(z)}(\tau_{-\lambda}(u))},\quad u\in \C^n. 
\end{align*}
This together with \eqref{toeplitzform}, \eqref{COV}, and \eqref{reprod} gives
\begin{align*}
\is{(\TP{\Phi}k_\lambda\otimes g) (z)}{h}
&=\int_{\C^n}\is{\Phi(u)g}{h}k_\lambda(u)\overline{k_z(u)}\D\gm(u)\\
&=\int_{\C^n}\is{\Phi(u)g}{h}\widetilde k_\lambda(u) k_\lambda(z) \overline{k_{\tau_{-\lambda}(z)}(\tau_{-\lambda}(u))}\D\gm(u)\\
&=k_\lambda(z) \int_{\C^n}\is{\Phi_\lambda(u)g}{h}  \overline{k_{\tau_{-\lambda}(z)}(u)}\D\gm(u)\\
&=k_\lambda(z) \int_{\C^n}\is{\Phi_\lambda(u)g}{(k_{\tau_{-\lambda}(z)}\otimes h) (u)} \D\gm(u)\\
&=k_\lambda(z) \is{\big(\PROJ(\Phi_\lambda\otimes g)\big)( \tau_{-\lambda}(z))}{h},
\end{align*}
which completes the proof.
\end{proof}
Having the above result, we can show that the adjoint of a vectorial Toeplitz operator is a suboperator of a vectorial integral operator with translation-related kernel. For a general studies of the adjoints of vectorial Toeplitz operators we refer the reader to \cite{cic-sm-2002, cic-sha-ms-2003}
\begin{pro}\label{lem2}
Suppose $\KHH\subseteq\dom{\TP{\Phi}}$. Then $\dom{\TP{\Phi}^*}\subseteq \dom{\IGR{\Theta_\Phi}}$ and
\begin{align*}
\TP{\Phi}^*F =\IGR{\Theta_\Phi} F,\quad F\in\dom{\TP{\Phi}^*},
\end{align*}
where  $\Theta_\Phi\colon \C^n\times\C^n\to\ogr{\hh}$ is the kernel given by
\begin{align*}
\Theta_\Phi (z,w)^*g={k_z(w)} \big(\PROJ (\Phi_z\otimes g)\big)(\tau_{-z}(w)),\quad g\in\hh,\ z,w\in\C^n.
\end{align*}.
\end{pro}
\begin{proof}
Let $F\in\dom{\TP{\Phi}^*}$. Then, by \eqref{toeplitzform} and Lemma \ref{lem1} we have\allowdisplaybreaks
\begin{align*}
\int_{\C^n}\is{\Theta_{\Phi}(z,w)F(w)}{g}\D\gm(w)
&=\int_{\C^n}\is{F(w)}{\Theta_{\Phi}(z,w)^* g}\D\gm(w)\\
&=\int_{\C^n}\is{F(w)}{(\TP{\Phi} k_z\otimes g)(w)}\D\gm(w)\\
&=\is{\TP{\Phi}^*F}{k_z\otimes g},\quad g\in\hh, z\in\C^n, 
\end{align*}
which means that $J_\Theta F\colon \C^n\to\hh$ is well-defined (the integral exists in the weak sense). Also, by the above computations, we have
\begin{align*}
\is{(J_{\Theta_\Phi} F) (z)}{g}=\is{(\TP{\Phi}^*F)(z)}{g},\quad g\in \hh, z\in\C^n.
\end{align*}
This implies that $(J_{\Theta_\Phi} F) (z)=(\TP{\Phi}^*F)(z)$ for every $z\in \C^n$. As a consequence, $J_{\Theta_\Phi} F\in \SB{L}{\hh}$ and $\IGR{\Theta_\Phi} F=\TP{\Phi}^*F$, which completes the proof. 
\end{proof}
\section{Main results}
Generalizing Stroethoff's result we use various kernels of related to the formula for the adjoint of a Toeplitz operator that was presented in Proposition \ref{lem2}. The following lemma provides a sufficient condition for one of those kernels to be $\hs{\hh}$-valued.
\begin{lem}\label{inths2combo}
Suppose $\KHH\subseteq\dom{\TP{\Phi}}$ and $z,w\in\C^n$. Let $\Xi_\Phi(z,w)\in\ogr{\hh}$ be given by
\begin{align*}
\Xi_\Phi(z,w) g&=\int_{\C^n} \Phi_z(u) g\,\overline{k_w(u)}\, \D\gm(u)=\PROJ(\Phi_z\otimes g)(w),\quad g\in\hh.
\end{align*}
Then the following conditions are satisfied$:$
\begin{itemize}
\item[(i)] $\hsn{\Xi_\Phi(z,w)}^2\leqslant k_w(w) \int_{\C^n}  \hsn{\Phi_z(u)}^2\D\gm(u)$,
\item[(ii)] if $\Phi$ is analytic, then $\hsn{\Xi_\Phi(z,w)}= \hsn{\Phi_z(w)}$. 
\end{itemize}
\end{lem}
\begin{proof}
Let $\{g_i\}_{i=1}^\infty$ an orthonormal basis for $\hh$.

(i) Assume that $\int_{\C^n}  \hsn{\Phi_z(u)}^2\D\gm(u)<\infty$. Then for $\gm$-a.e. $u\in \C^n$, $\Phi_z(u)$ is Hilbert-Schmidt. Then we get\allowdisplaybreaks
\begin{align*}
\sum_{i,j=1}^\infty |\is{\Xi_\Phi(z,w)g_i}{g_j}|^2
&=\sum_{i,j=1}^\infty \Big|\int_{\C^n}\is{\Phi_z(u)g_i}{g_j}\overline{k_w(u)}\D\gm(u)\Big|^2\\
&\leqslant \sum_{i,j=1}^\infty \bigg(\int_{\C^n}|\is{\Phi_z(u)g_i}{g_j}|^2\D\gm(u)\bigg)\bigg(\int_{\C^n}|k_w(u)|^2\D\gm(u)\bigg)\\
&= \|k_w\|^2\int_{\C^n}  \Big(\sum_{i,j=1}^\infty |\is{\Phi_z(u)g_i}{g_j}|^2\Big)\D\gm(u)\\
&=k_w(w)\int_{\C^n}\hsn{\Phi_z(u)}^2\D\gm(u),
\end{align*}
which gives (i).

(ii) Due to the reproducing property we have
\begin{align*}
\is{\PROJ\big(\Phi_z\otimes g_i\big)(w)}{g_j}&=\int_{\C^n} \is{\Phi_z(u)g_i}{g_j}\overline{k_w(u)}\, \D\gm(u)=\is{\Phi_z(w)g_i}{g_j},\quad i,j\in\N.
\end{align*}
Therefore we get\allowdisplaybreaks
\begin{align*}
\sum_{i,j=1}^\infty |\is{\Xi_\Phi(z,w)g_i}{g_j}|^2
&=\sum_{i,j=1}^\infty |\is{\Phi_z(w)g_i}{g_j}|^2,
\end{align*}
which proves (ii) and completes the proof.
\end{proof}
Employing Lemma \ref{inths} we get.
\begin{cor}
Suppose $\KHH\subseteq\dom{\TP{\Phi}}$. Suppose that $\Phi$ is analytic and kernel $(z,w)\mapsto\Phi(z+w)$ is Hilbert-Schmidt. Then the integral operator $\IGR{\Xi_\Phi}$ with  $\Xi_\Phi$ as in Lemma \ref{inths2combo} is Hilbert-Schmidt on $\mathcal{B}\otimes \hh$.
\end{cor}
As shown below, the $\ogr{\hh}$-valued kernel of the adjoint of a Toeplitz operator multiplied by appropriate complex function induces a Hilbert-Schmidt integral operator.
\begin{pro}\label{inths4new}
Let $\KHH\subseteq\dom{\TP{\Phi}}$ and  $\Xi_\Phi\colon\C^n\times\C^n\to\ogr{\hh}$ be as in Lemma \ref{inths2combo}. Suppose that $\gamma\colon\C^n\times \C^n\to \C$ is Borel and the function $\alpha\colon\C^n\to\overline{\R}_+$ given by
\begin{align*}
\alpha(z)=\int_{\C^n} |\gamma(z,w)|^2\hsn{\Xi_\Phi(z,w)}^2\D\gm(w), \quad z\in\C^n,
\end{align*}
belongs to $L^1(V)$. Then the integral operator $\IGR{\Theta_{\gamma,\Phi}}$ with  $\Theta_{\gamma,\Phi}\colon\C^n\times \C^n\to\ogr{\hh}$ given by
\begin{align*}
\Theta_{\gamma,\Phi}(x,y)^*g=\gamma(x,\tau_{-x}(y))\,{k_x(y)}\,\big(\PROJ \Phi_x\otimes g\big)(\tau_{-x}(y)),\quad g\in\hh,\ x,y\in\C^n,
\end{align*}
is Hilbert-Schmidt on $\mathcal{B}\otimes \hh$. 
\end{pro}
\begin{proof}
According to Lemma \ref{inths} the proof amounts to showing that the kernel $\Theta_{\gamma,\Phi}$ is Hilbert-Schmidt. Since, by the Fubini's theorem and \eqref{COV}, we have
\begin{align*}
\int_{\C^{2n}} \hsn{\gamma(z,\tau_{-z}(w))k_z(w)&\Xi_\Phi(z,\tau_{-z}(w))}^2\D\gm(z,v)\\
&=\int_{\C^{2n}} |\gamma(z,w)k_z(\tau_z(w))|^2\widetilde k_{-z}(w)\hsn{\Xi_\Phi(z,w)}^2\D\gm(z,w)\\
&=\int_{\C^n} k_z(z)\bigg( \int_{\C^n} |\gamma(z,w)|^2\hsn{\Xi_\Phi(z,w)}^2\D\gm(w)\bigg)\D\gm(z),
\end{align*}
the claim follows.
\end{proof}
\begin{cor}
Let $\Phi\in L^2\big(\gm; \ogr{\hh}\big)$. Let $\varDelta\subseteq\C^n$ be a bounded Borel set. If 
\begin{align*}
\int_\varDelta\int_{\C^n} \hsn{\Phi_z(u)}^2\D\gm(u)\D V(z)<\infty,
\end{align*}
then the integral operator $\IGR{\Theta_{\chi_{\varDelta\times\varDelta},\Phi}}$, given as in Proposition \ref{inths4new}, is Hilbert-Schmidt on $\mathcal{B}\otimes \hh$. 
\end{cor}
\begin{proof}
In view of Proposition \ref{inths4new}, $\IGR{\Theta_{\chi_{\varDelta\times\varDelta},\Phi}}$ is Hilbert-Schmidt whenever
\begin{align}\label{pdc}
\int_{\C^n}\int_{\C^n}\chi_{\varDelta\times\varDelta}(z,w)\hsn{\Xi_\Phi(z,w)}^2\D\gm(w)\D V(z)<\infty.
\end{align}
Since, by Lemma \ref{inths2combo}\,(i), we have
\begin{align*}
\int_{\C^n}\int_{\C^n}\chi_{\varDelta\times\varDelta}(z,w)\hsn{\Xi_\Phi(z,w)}^2\D\gm(w)\D V(z)
\leqslant V(\varDelta)\int_{\varDelta} \int_{\C^n}\hsn{\Phi_z(u)}^2\D\gm(u)\D V(z),
\end{align*}
inequality \eqref{pdc} follows from the assumptions.
\end{proof}
\begin{cor}\label{venti}
Let $\Phi\in L^\infty\big(\gm; \hs{\hh}\big)$. Let $\varDelta\subseteq\C^n$ be a bounded Borel set. Then the integral operator $\IGR{\Theta_{\gamma_\varDelta,\Phi}}$ with $\gamma_\varDelta\colon \C^n\times\C^n\to\R$ given by $\gamma_\varDelta:=\chi_{\varDelta\times\C^n}$ is Hilbert-Schmidt on $\mathcal{B}\otimes \hh$. 
\end{cor}
\begin{proof}
First we show that
\begin{align}\label{dolcelatte}
\hsn{\Xi_\Phi(z,w)}^2\leqslant C\,\exp{\frac{|w|^2}{2}},
\end{align}
where $C>0$ is a constant such that $\hsn{\Phi(u)}\leqslant C$ for $\gm$-a.e. $u\in \C^n$. Indeed, let $\{g_i\}_{i=1}^\infty$ be an orthonormal basis for $\hh$. Then, applying Jensen's inequality with a probabilistic measure $\D\nu_{z,w}(t)=|k_w(t-z)|\exp\Big(-|t-z|^2-\frac{|w|^2}{4}\big)\D V(t)$,  we get\allowdisplaybreaks
\begin{align*}
\sum_{i,j=1}^\infty |\is{\Xi_\Phi(z,w)g_i}{g_j}|^2
&=\sum_{i,j=1}^\infty \Big|\int_{\C^n}\is{\Phi_z(u)g_i}{g_j}\overline{k_w(u)}\D\gm(u)\Big|^2\\
&\leqslant \sum_{i,j=1}^\infty \Big(\int_{\C^n}|\is{\Phi_z(u)g_i}{g_j}||k_w(u)|\D\gm(u)\Big)^2\\
&= \sum_{i,j=1}^\infty \Big(\int_{\C^n}|\is{\Phi(t)g_i}{g_j}||k_w(t-z)|\exp\big(-|t-z|^2\big) \D V(t)\Big)^2\\
&= \sum_{i,j=1}^\infty \Big(\int_{\C^n}|\is{\Phi(t)g_i}{g_j}|\exp\Big(\frac{|w|^2}{4}\Big) \D\nu_{z,w}(t)\Big)^2\\
&\leqslant \sum_{i,j=1}^\infty \int_{\C^n}|\is{\Phi(t)g_i}{g_j}|^2\exp\Big(\frac{|w|^2}{2}\Big) \D\nu_{z,w}(t)\\
&= \int_{\C^n}\hsn{\Phi(t)}^2\exp\Big(\frac{|w|^2}{2}\Big) \D\nu_{z,w}(t),
\end{align*}
which yields \eqref{dolcelatte}. Consequently, we deduce that
\begin{align*}
\int_{\C^n} k_z(z)\bigg( \int_{\C^n} |\gamma_\varDelta(z,w)|^2&\hsn{\Xi_\Phi(z,w)}^2\D\gm(w)\bigg)\D\gm(z)\\
&=\int_{\varDelta} k_z(z)\bigg( \int_{\C^n} \hsn{\Xi_\Phi(z,w)}^2\D\gm(w)\bigg)\D\gm(z)\\
&\leqslant C\int_\varDelta k_z(z) \int_{\C^n}\exp\bigg(\frac{|w|^2}{2}\bigg)\D\gm(w)\D\gm(z)<\infty.
\end{align*}
This implies that $\IGR{\Theta_{\gamma_\varDelta,\Phi}}$ is Hilbert-Schmidt (see the proof of Proposition \ref{inths4new}).
\end{proof}

Having the the above result, we are able now to provide a sufficient condition for the compactness of $\IGR{\Phi}$ \'{a} la Stroethoff.
\begin{thm}\label{main}
Let $\Phi\in L^\infty\big(\gm; \hs{\hh}\big)$. Suppose that$:$
\begin{align}\label{stendahl}
\lim_{|z|\to\infty}\int_{\C^n}\hsn{\Xi_\Phi(z,w)}^2\D\gm(w)=0.    
\end{align}
Then $\TP{\Phi}$ is compact.
\end{thm}
\begin{proof}
For a given $r>0$ we define a function $\gamma_r\colon \C^n\times\C^n\to\R$ by $\gamma_r(z,w)=\chi_{\{|z|<r\}\times \C^n}$, where as usual $\{|z|<r\}:=\{z\in\C^n \colon |z|<r\}$. Then, by Corollary \ref{venti}, the integral operator $\IGR{\Theta_{\gamma_r,\Phi}}$ is Hilbert-Schmidt for every $r>0$. Hence, showing that $\IGR{\Theta_{\gamma_r,\Phi}}\to \TP{\Phi}^*$ as $r\to\infty$ yields the claim. 

Observe that by Proposition \ref{lem2} we have
\begin{align}\label{paszporty}
\TP{\Phi}^*-\IGR{\Theta_{\gamma_r,\Phi}}=\IGR{\Theta_\Phi}-\IGR{\Theta_{\gamma_r,\Phi}}=\IGR{\Theta^{(r)}},
\end{align}
where
\begin{align*}
\Theta^{(r)}(z,w)=\Theta_{\eta_r,\Phi}(z,w),\quad z,w\in\C^n,
\end{align*}
with $\eta_r=\chi_{\{|z|\geqslant r\}\times \C^n}$. We will evaluate the norm of $\IGR{\Theta^{(r)}}$ using the Schur test. First, we note that there exists a constant $\alpha>0$ such that
\begin{align}\label{sczur1}
\int_{\C^n}\eta_r(z)|k_z(w)|\hsn{\Xi_{\Phi}(z,\tau_{-z}(w))}k_z(z)^{\frac12}\D\gm(z)\leqslant \alpha\, k_w(w)^{\frac12},\quad w\in\C^n.
\end{align}
Indeed, by \eqref{dolcelatte}, we have
\begin{align*}
\int_{\C^n}\eta_r(z)|k_z(w)|\hsn{\Xi_{\Phi}(z,\tau_{-z}(w))}&k_z(z)^{\frac12}\D\gm(z)
\leqslant C\int_{\C^n} |k_z(w)|\exp\Big(\frac{|\tau_{-z}(w)|}{4}\Big)k_z(z)^{\frac12}\D\gm(z)\\
&=C\,k_w(w)^{\frac14}\int_{\C^n}|k_z(w)|^{\frac12}\exp\Big(-\frac{|z|^2}{4}\Big)\D V(z).
\end{align*}
Since the last integral equals $2^n(2\pi)^n k_w(w)^{\frac{1}{4}}$, we get \eqref{sczur1}. On the other hand, there exists a function $\beta\colon \R_+\to\R_+$ such that 
\begin{align}\label{sczur1.5}
\lim_{r\to\infty}\beta(r)=0    
\end{align}
and
\begin{align}\label{sczur2}
\int_{\C^n}\eta_r(z)|k_z(w)|\hsn{\Xi_{\Phi}(z,\tau_{-z}(w))}k_w(w)^{\frac12}\D\gm(w)\leqslant \beta(r)\, k_z(z)^{\frac12},\quad z\in\C^n.
\end{align}
This can be proved as follows. By the change-of-variable theorem we have
\begin{align*}\allowdisplaybreaks
\int_{\C^n}\eta_r(z)|k_z(w)|&\hsn{\Xi_{\Phi}(z,\tau_{-z}(w))}k_w(w)^{\frac12}\D\gm(w)\\
&= \int_{\C^n}\eta_r(z)|k_z(\tau_z(t))|\hsn{\Xi_{\Phi}(z,t)}k_{\tau_z(t)}(\tau_z(t))^{\frac12}\exp\big(-|\tau_z(t)|^2\big) \D V(t)\\
&=\eta_r(z) k_z(z)^{\frac12}\int_{\C^n}\hsn{\Xi_{\Phi}(z,t)}\exp\Big(-\frac{|t|^2}{2}\Big) \D V(t) .
\end{align*}
Since $\exp\Big(-\frac{|t|^2}{2}\Big)=\exp\Big(-\frac{5|t|^2}{12}\Big) \exp\Big(-\frac{|t|^2}{12}\Big)$, applying H\"older's inequality with conjugate exponents 3 and $\frac{3}{2}$ and later \eqref{dolcelatte} we see that
\begin{align*}
\int_{\C^n}\eta_r(z)&|k_z(w)|\hsn{\Xi_{\Phi}(z,\tau_{-z}(w))}k_w(w)^{\frac12}\D\gm(w)\\
&\leqslant
\eta_r(z) k_z(z)^{\frac12}\bigg(\int_{\C^n}\hsn{\Xi_{\Phi}(z,t)}^3\exp\Big(-\frac{5|t|^2}{4}\Big) \D V(t)\bigg)^\frac{1}{3}\bigg( \int_{\C^n} \exp\Big(-\frac{|t|^2}{8}\Big)\D V(t)\bigg)^{\frac23}\\
&\leqslant
\eta_r(z) k_z(z)^{\frac12} \Big(2 (2\pi)^n\Big)^{\frac23} C^{\frac16}\bigg(\int_{\C^n}\hsn{\Xi_{\Phi}(z,t)}^2\D\gm(t)\bigg)^\frac{1}{3}.
\end{align*}
Hence, in view of \eqref{stendahl}, we deduce \eqref{sczur1.5} and \eqref{sczur2} with 
\begin{align*}
\beta(r)=\big(2 (2\pi)^n\big)^{\frac23} C^{\frac16}\sup\limits_{\{|z|\geqslant r\}}\bigg(\int_{\C^n} \hsn{\Xi_{\Phi}(z,t)}^2\D\gm(t)\bigg)^{\frac13}, \quad r\in\R_+    
\end{align*}
Now, using the Schur test (see \cite[Theorem 5.2]{hal-sun}), \eqref{sczur1}, and \eqref{sczur2} we get
\begin{align*}
\|\IGR{\Theta^{(r)}}\|^2\leqslant \alpha\,\beta(r),\quad r\in\R_+.
\end{align*}
Therefore, by \eqref{sczur1.5} and \eqref{paszporty}, $\TP{\Phi}$ is compact.
\end{proof}
It remains an open problem as to whether condition \eqref{stendahl} is necessary for compactness of $\TP{\Phi}$ with $\hs{\hh}$-valued symbols. Adapting the methods from the classical complex-valued case we get a weaker condition.
\begin{pro}\label{necc}
Let $\TP{\Phi}\in\ogr{\mathcal{B}\otimes \hh}$ be compact. Then
\begin{align}\label{rindo}
\lim_{|z|\to\infty}\int_{\C^n}\|(\PROJ \Phi_z\otimes g)(w)\|^2\D\gm(w)=0,\quad g\in \hh.
\end{align}
\end{pro}
\begin{proof}
Applying twice Lemma \ref{lem1} we get
\begin{align*}
\|\TP{\Phi} (k_z\otimes g)\|^2
=\int_{\C^n}|k_z(w)|^2\|\PROJ(\Phi_z\otimes g)(\tau_{-z}(w))\|^2\D\gm(w),\quad z\in\C^n, g\in\hh.
\end{align*}
Combining this with \eqref{COV} and
\begin{align}\label{hite}
\|k_z \otimes g\|=\|g\| \exp{|z|^2},\quad z\in\C^n, g\in\hh,
\end{align}
yields
\begin{align}\notag
\bigg\|\TP{\Phi} \bigg( \frac{k_z\otimes g}{\|k_z\otimes g\|}\bigg)\bigg\|^2
&=\frac{1}{\|g\|^2}\int_{\C^n}\|\PROJ(\Phi_z\otimes g)(\tau_{-z}(w))\|^2 \widetilde k_z(w) \D\gm(w)\\
&=\frac{1}{\|g\|^2}\int_{\C^n}\|\PROJ(\Phi_z\otimes g)(w)\|^2 \D\gm(w),\quad z\in\C^n, g\in\hh. \label{tnorm}
\end{align}
The compactness of $\TP{\Phi}$ implies that $\|\TP{\Phi} x_\iota\|$ converges to $0$ for any sequence $\{x_\iota\}$ converging to $0$ weakly. In particular, we deduce from \eqref{hite} and
\begin{align*}
|\is{p\otimes h}{k_z\otimes g}|=|\is{p(z) h}{g}|\leqslant |p(z)| \|h\|\|g\|,\quad z\in \C^n, p\in\PLN, h,g\in \hh, 
\end{align*}
that
\begin{align*}
\lim_{|z|\to\infty} \bigg\|\TP{\Phi} \bigg( \frac{k_z\otimes g}{\|k_z\otimes g\|}\bigg)\bigg\|^2=0. 
\end{align*}
This, in view of \eqref{tnorm}, proves the claim.
\end{proof}
It is worth noting that there is quite a gap between conditions \eqref{stendahl} and \eqref{rindo} as shown in the following examples. The first of them shows that even in a very simplified situation the integrals in \eqref{stendahl} and \eqref{rindo} differ essentially and hints the importance of using the Hilbert-Schmidt norm.
\begin{exa}\label{star1}
Let $\hh$ be an infinite dimensional Hilbert space and $A\in\ogr{\hh}$ be an operator that is not Hilbert-Schmidt. Define $\Phi\colon \C^n\to\ogr{\hh}$ by $\Phi(z)=A$. Clearly, $\Phi$ is analytic,
\begin{align*}
\int_{\C^n} \|\PROJ(\Phi_z\otimes g)(w)\|^2\D\gm(w)
&=\int_{\C^n} \|Ag\|^2\D\gm(w)<\infty,\quad z\in\C^n, g\in\hh,
\end{align*}
and
\begin{align*}
\int_{\C^n}\|\Xi_\Phi(z,w)\|^2\D\gm(w)=\int_{\C^n} \|A\|^2\D\gm(w)<\infty,\quad z\in\C^n.
\end{align*}
On the other hand, for an orthonormal basis $\{g_i\}_{i=1}^\infty$ of $\hh$ we have
\begin{align*}
\hsn{\Xi_\Phi(z,w)}^2=\sum_{i,j=1}^\infty \big|\is{\Phi_z(w)g_j}{g_j}\big|^2=\hsn{A}^2=\infty,\quad z,w\in\C^n,
\end{align*}
which implies that
\begin{align*}
\int_{\C^n}\hsn{\Xi_{\Phi}(z,w)}^2\D\gm(w)=\infty,\quad z\in \C^n.
\end{align*}
Note that $\TP{\Phi}$ is not compact.
\end{exa}
The second example provides a non-compact vectorial Toeplitz operator which satisfies \eqref{rindo}. Its symbol is not $\hs{\hh}$-valued and \eqref{stendahl} does not hold.
\begin{exa}\label{star2}
Let $\hh$ be an infinite dimensional Hilbert space. Let $\phi\colon\C^n\to\C$ be a non-zero function such that the classical Toeplitz operator $\TP{\phi}$ on $\mathcal{B}$ is compact. In view of \cite[Theorem 5]{str-mmj-1992} we necessarily have
\begin{align}\label{bluenote}
\lim_{|z|\to\infty} \int_{\C^n} |\mathrm{P}(\phi\circ\tau_z)(w)|^2 \D\gm(w)=0,
\end{align}
where $\mathrm{P}$ denotes an orthogonal projection from $L^2(\gm)$ onto $\mathcal{B}$. This implies that for a function $\Phi\colon\C^n\to\ogr{\hh}$ given by $\Phi(z)=\phi(z)\,I$, $I$ being the identity operator on $\hh$, we have
\begin{align*}
\int_{\C^n}\|\PROJ\big(\Phi_z\otimes g\big)(w)\|^2 \D\gm(w)
&=\int_{\C^n}\|\PROJ\big(\phi_z\otimes g\big)(w)\|^2 \D\gm(w)\\
&=\int_{\C^n}|\mathrm{P}(\phi\circ \tau_z)\big)(w)|^2 \|g\|^2\D\gm(w),\quad z\in\C^n,g\in\hh,
\end{align*}
and thus by \eqref{bluenote} we have
\begin{align*}
\lim_{|z|\to\infty}\int_{\C^n}\|\PROJ\big(\Phi_z\otimes g\big)(w)\|^2 \D\gm(w)=0,\quad g\in\hh.
\end{align*}
On the other hand, since the identity operator on an infinite dimensional Hilbert space is not compact, we have
\begin{align*}
\int_{\C^n}\hsn{\Xi_{\Phi}(z,w)}^2\D\gm(w)
&=\int_{\C^n}|\mathrm{P}(\phi\circ\tau_z)(w))|^2\hsn{I}^2\D\gm(w) =\infty,\quad z\in \C^n.
\end{align*}
Clearly, $\TP{\Phi}$ is not compact as the image of $\{\chi_{\C^n}\otimes e_i\colon i\in\N\}$ via $\TP{\Phi}$, where $\{e_i\colon i\in\N\}$ is an orthonormal basis of $\hh$, does not contain a convergent subsequence.
\end{exa}
The next example, which was pointed to us be an anonymous referee, is showing that condition \eqref{stendahl} may characterize the compactness of $\TP{\Phi}$ only in the case when $\Phi$ is $\hs{\hh}$-valued. The latter is not necessary in general. 
\begin{exa}\label{taebaek}
We modify the symbol $\Phi$ from Example \ref{star2} by replacing the identity operator by any compact operator $A$ which is not Hilbert-Schmidt, i.e., $\Phi(z)=\phi(z)A$, $z\in\C^n$, where $\phi\colon\C^n\to\C$ is a non-zero function such that $\TP{\phi}$ on $\mathcal{B}$ is compact. Then \eqref{rindo} is satisfied while \eqref{stendahl} is not, just as in Example \eqref{star2}. Nonetheless, $\TP{\Phi}$ is compact. Indeed, since $\TP{\phi}$ and $A$ are compact for any bounded sequences $\{f_i\}_{i=1}^\infty\subseteq L^2(\gm)$ and $\{g_j\}_{j=1}^\infty\subseteq \hh$ there are subsequences $\{f_{i_k}\}_{k=1}^\infty$ and $\{g_{j_k}\}_{k=1}^\infty$ such that $\{\TP{\phi} f_{i_k}\}_{k=1}^\infty$ and $\{Ag_{j_k}\}_{k=1}^\infty$ which are convergent. This implies that $\{\TP{\Phi} f_{i_k}\otimes g_{j_k}\}_{k=1}^\infty$ is convergent. Since the simple tensors $\{f\otimes g\colon f\in L^2(\gm), g\in\hh\}$ are dense in $L^2(\gm)\otimes\hh$, we deduce that $\TP{\Phi}$ is compact.
\end{exa}
We conclude the paper with bringing up to the reader's attention the questions that remains open. The first asks whether there exists a compact $\TP{\Phi}$ induced by a $\hs{\hh}$-valued symbol $\Phi$ for which \eqref{stendahl} is not satisfied. The second asks if there is a $\hs{\hh}$-valued symbol $\Phi$ such that \eqref{rindo} holds but $\TP{\Phi}$ is not compact.
\section*{Acknowledgments}
The first author was supported by the Polish National Science Centre (NCN) Opus grant no. 2017/25/B/ST1/00906. The second author was supported by the SONATA BIS grant no. UMO- 2017/26/E/ST1/00723 financed by the National Science Centre, Poland. The third named author was supported by Basic Science Research Program through the National Research Foundation of Korea(NRF) funded by the Ministry of Education(NRF-2015R1C1A1A01053837)

A substantial part of the paper was written while the first and the second author were visiting the Department of Mathematics of Chungnam National University during the autumn of 2018. They wish to thank the faculty and the administration of the unit for warm hospitality.

We would like to thank an anonymous referee for many valuable comments which helped to improve the paper, in particular for pointing out a symbol $\Phi$ of Example \ref{taebaek}.
\bibliographystyle{amsalpha}

\end{document}